\documentclass[10pt, leqno]{article}

\usepackage{amsmath,mathrsfs,amsthm}
\usepackage{txfonts}

\DeclareMathOperator{\rank}{\mathrm{rank}}
\DeclareMathOperator{\End}{\mathrm{End}}

\newtheorem{theorem}{Theorem}
\newtheorem{lemma}[theorem]{Lemma}
\newtheorem{corollary}[theorem]{Corollary}
\newtheorem{proposition}[theorem]{Proposition}

\theoremstyle{definition}
\newtheorem{definition}[theorem]{Definition}

\newtheorem{remark}[theorem]{Remark}
\newtheorem{example}[theorem]{Example}

\newcommand{\g}{\mathfrak{g}}
\newcommand{\s}{\mathfrak{s}}
\newcommand{\C}{\mathbb{C}}
 \newcommand{\A}{\mathcal{A}}
 \newcommand{\Sp}{\operatorname{Sp}}
 \newcommand{\SL}{\operatorname{SL}}
\newcommand{\GL}{\operatorname{GL}}
\newcommand{\SO}{\operatorname{O}}

\begin{document}

\title{A complete set of intertwiners for arbitrary tensor product representations via current algebras}
\author{Shrawan Kumar}
\date{}

\maketitle

{\it Abstract:} Let $\mathfrak{g}$ be a reductive Lie algebra and let $\vec{V}(\vec{\lambda})$ be a tensor product of $k$ copies of finite dimensional irreducible $\g$-modules. Choosing $k$ points in $\C$, $\vec{V}(\vec{\lambda})$  acquires a natural structure of the current algebra $\g\otimes \C[t]$-module.  Following a work of Rao [R], we produce an explicit and complete set of $\g$-module intertwiners of $\vec{V}(\vec{\lambda})$ in terms of the action of the current algebra. 

\section{Introduction}

Let $\mathfrak{g}$ be a finite dimensional reductive Lie algebra over the complex numbers $\mathbb{C}$
and let $A$  be any commutative $\C$-algebra with identity. Then, $\mathfrak{g}\otimes A$ acquires a natural structure of Lie algebra.
Take any $\mathfrak{g}$-invariant $\theta\in \left[\mathfrak{g}^{\otimes k}\right]^{\mathfrak{g}}$, $\theta=\sum\limits_{i}x^{i}_{1}\otimes x^{i}_{2}\otimes\cdots\otimes x^{i}_{k}$,  and any $P_{1},\dots,P_{k}\in A$. 
Then, $\theta(P_{1},\dots,P_{k})\in U(\mathfrak{g}\otimes A)$ defined by
$$
\theta(P_{1},\dots,P_{k}):=\sum\limits_{i}x^{i}_{1}(P_{1})\dots x^{i}_{k}(P_{k}),$$
commutes with $\g$ (cf. Lemma \ref{lem6}), where $U(\mathfrak{g}\otimes A)$ is the enveloping algebra of $\mathfrak{g}\otimes A$.

In fact, we show that, via the above construction, we get all of $\left[U(\mathfrak{g}\otimes A)\right]^\g$ as $\theta$ ranges over  $\left[\mathfrak{g}^{\otimes k}\right]^{\mathfrak{g}}$ and $P_i$'s range over the elements of $A$ (see Proposition \ref{prop8} for a more precise result).

Let us take $A= \mathbb{C}[t]$ and denote $\g\otimes A$ by $\g[t]$. For any $\vec{p}= (p_1, \dots, p_d) \in \C^d$ and any irreducible $\g$-modules $V(\lambda_1), \dots, V(\lambda_d)$ with highest weights $\lambda_1, \dots, \lambda_d$ respectively, consider the tensor product $\g$-module 
$$\vec{V}(\vec{\lambda}) := V(\lambda_1) \otimes \dots \otimes V(\lambda_d), \,\,\,\text{where}\,\, \vec{\lambda}= (\lambda_1, \dots, \lambda_d).$$
Then, $\vec{V}(\vec{\lambda})$ acquires the structure of a $\g[t]$-module (called an {\it evaluation module}):
$$x(P)\cdot (v_1\otimes \dots \otimes v_d):= \sum_{i=1}^d\, P(p_i) v_1\otimes \dots \otimes x\cdot v_i\otimes \dots \otimes v_d,\,\,\,\text{for}\,\, x\in \g, P\in \C[t], v_i \in V(\lambda_i).$$
To emphasize the $\g[t]$-module structure, we denote $\vec{V}(\vec{\lambda})$ by $\vec{V}_{\vec{p}}(\vec{\lambda})$.

Consider the Casimir element  $\Omega \in \left[\g^{\otimes 2}\right]^\g$. Then, we calculate the action of $\Omega (P, Q)$ on 
 $\vec{V}_{\vec{p}}(\vec{\lambda})$ for any $\vec{\lambda}=(\lambda_1, \lambda_2)$ (cf. Lemma \ref{casimir}). 

 We assume now that $p_i$ are all distinct. In this case $\vec{V}_{\vec{p}}(\vec{\lambda})$ is an irreducible $\g[t]$-module.
Decompose $\vec{V}_{\vec{p}}(\vec{\lambda})$  into its isotypic components (as a $\g$-module):
$$
\vec{V}_{\vec{p}}(\vec{\lambda}) =\oplus \vec{V}_{\vec{p}}(\vec{\lambda}) [\mu],
$$
where $\vec{V}_{\vec{p}}(\vec{\lambda}) [\mu]$ denotes the  isotypic component corresponding to the highest weight $\mu$. 

Clearly, the action of $\g$ commutes with the action of $\left[U(\g[t])\right]^\g$  on $\vec{V}_{\vec{p}}(\vec{\lambda})$. Thus, we get an action of 
$\g\times \left[U(\g[t])\right]^\g$ on $\vec{V}_{\vec{p}}(\vec{\lambda})$  stabilizing each isotypic component $\vec{V}_{\vec{p}}(\vec{\lambda}) [\mu].$ 

The following theorem is one of our principal results of the paper (cf. Theorem \ref{thm9}).
\begin{theorem}
Each isotypic component $\vec{V}_{\vec{p}}(\vec{\lambda}) [\mu]$ is an irreducible module for $\mathfrak{g}\times \left[U(\g[t])\right]^\g$.
\end{theorem}

In Sections 3, 5 and 6, we determine the spaces $\left[\mathfrak{g}^{\otimes k}\right]^{\mathfrak{g}}$ for $\g = gl(n), {\s}p(2n)$ and ${\s}o(n)$ respectively using the First Fundamental Theorem of Invariant Theory. (In the last case, for $n$ even, we only determine $\left[\mathfrak{g}^{\otimes k}\right]^{\SO(n)}$.)

Let $V =\C^n$ be the standard representation of $\g=gl(n)$. For any positive integer $k$, the symmetric group $\Sigma_k$ acts on the tensor product $V^{\otimes k}$ by permuting the factors. Clearly, this action of $\Sigma_k$ commutes with the tensor product action of $gl(n)$
on $V^{\otimes k}$. 
Thus, we have an algebra homomorphism:
$$\Phi: \C[\Sigma_k] \to \End_\g(V^{\otimes k}),$$
where $\End_\g(V^{\otimes k})$ denotes the space of $\g$-module endomorphisms of $V^{\otimes k}$.
By the Schur-Weyl duality,  the above map $\Phi$ is an (algebra) isomorphism.
Choose  $\vec{p}=(p_1, \dots , p_k) \in \C^k$ such that $p_i$'s are distinct.  There is a surjective algebra homomorphism (for $\g=gl(n)$)
$$\varphi^{o}: \left[U(\mathfrak{g}[t])\right]^{\mathfrak{g}}\twoheadrightarrow \End_{\mathfrak{g}}(V^{\otimes k}) \simeq  \C[\Sigma_k] ,$$
where the last identification is via $\Phi$ (cf.  proof of Theorem \ref{thm9}). Thus, we get a surjective algebra homomorphism
$$\Xi: \left[U(\mathfrak{g}[t])\right]^{\mathfrak{g}}\twoheadrightarrow  \C[\Sigma_k] .$$ We give an explicit preimage of any reflection  $\tau = (r, s) \in \Sigma_k$ under $\Xi$ (cf. Proposition \ref{schur}).

\vskip2ex
\noindent
{\bf Acknowledgements.} I thank S. E. Rao for sharing his work [R]. His work and questions therein led to this work. 
 This work was supported partially by the NSF grant DMS-1501094.

\section{Intertwining operators - Main results}
Let $\mathfrak{g}$ be a finite dimensional reductive Lie algebra over the complex numbers $\mathbb{C}$
and let $A$  be any commutative $\C$-algebra with identity. Then, $\mathfrak{g}\otimes A$ acquires a natural structure of Lie algebra:
$$[x(P), y(Q)]:= [x,y](PQ), \,\,\,\text{for}\,\, x,y\in \mathfrak{g}, \, P,Q\in A,$$
where $x(P)$ denotes $x\otimes P$.

\begin{definition}\label{defi5}
Take any $\mathfrak{g}$-invariant (under the adjoint action) $\theta\in \left[\mathfrak{g}^{\otimes k}\right]^{\mathfrak{g}}$, $\theta=\sum\limits_{i}x^{i}_{1}\otimes x^{i}_{2}\otimes\cdots\otimes x^{i}_{k}$,  and any $P_{1},\dots,P_{k}\in A$. 
Define $\theta(P_{1},\dots,P_{k})\in U(\mathfrak{g}\otimes A)$ by
\begin{equation}
\theta(P_{1},\dots,P_{k}):=\sum\limits_{i}x^{i}_{1}(P_{1})\dots x^{i}_{k}(P_{k}),\label{eq12}
\end{equation}
where $U(\mathfrak{g}\otimes A)$ denotes the enveloping algebra of $\mathfrak{g}\otimes A.$
\end{definition}

\begin{lemma}\label{lem6}
$[\mathfrak{g},\theta(P_{1},\dots,P_{k})]=0$.
\end{lemma}

\begin{proof}
Let $\pi : T(\mathfrak{g}\otimes A)\twoheadrightarrow U(\mathfrak{g}\otimes A)$ be the canonical surjective homomorphism, where $T$ is the tensor algebra. Consider the element
$\hat{\theta} (P_{1},\dots,P_{k})\in T(\mathfrak{g}\otimes A)$ defined by
$$\hat{\theta} (P_{1},\dots,P_{k})=\sum_{i}x^{i}_{1}(P_{1})\otimes\dots\otimes x^{i}_{k}(P_{k})\in T(\mathfrak{g}\otimes A).$$
For any $y\in \mathfrak{g}$,
\begin{align}
[y,\hat{\theta}(P_{1},\dots,P_{k})] &= \sum\limits_{i}\sum\limits^{k}_{j=1}x^{i}_{1}(P_{1})\otimes\cdots\otimes [y,x^{i}_{j}](P_{j})\otimes\cdots\otimes x^{i}_{k}(P_{k})\notag\\[4pt]
&= 0,\text{~~ since~~ } [\mathfrak{g},\theta]=0,\label{eq15}
\end{align}
and $\mathfrak{g}(P_{1})\otimes\cdots\otimes \mathfrak{g}(P_{k})\simeq \mathfrak{g}\otimes \cdots\otimes \mathfrak{g}$ as $\mathfrak{g}$-modules under the adjoint action.

Now, since $\pi(\hat{\theta}(P_{1},\dots,P_{k}))=\theta(P_{1},\dots,P_{k})$, by the identity \eqref{eq15}, the lemma follows.
\end{proof}
\begin{definition}\label{defi6} (Evaluation modules) Let us take $A= \mathbb{C}[t]$ and denote it by $\A$. In this case, we denote $\g\otimes \A$ by $\g[t]$. For any $\vec{p}= (p_1, \dots, p_d) \in \C^d$ and any irreducible $\g$-modules $V(\lambda_1), \dots, V(\lambda_d)$ with highest weights $\lambda_1, \dots, \lambda_d$ respectively, consider the tensor product $\g$-module 
$$\vec{V}(\vec{\lambda}) := V(\lambda_1) \otimes \dots \otimes V(\lambda_d), \,\,\,\text{where}\,\, \vec{\lambda}= (\lambda_1, \dots, \lambda_d).$$
Then, $\vec{V}(\vec{\lambda})$ acquires the structure of a $\g[t]$-module:
$$x(P)\cdot (v_1\otimes \dots \otimes v_d):= \sum_{i=1}^d\, P(p_i) v_1\otimes \dots \otimes x\cdot v_i\otimes \dots \otimes v_d,\,\,\,\text{for}\,\, x\in \g, P\in \A, v_i \in V(\lambda_i).$$
To emphasize the $\g[t]$-module structure, we denote $\vec{V}(\vec{\lambda})$ by $\vec{V}_{\vec{p}}(\vec{\lambda})$ and it is called an {\it evaluation module}. It is well known (and easy to prove) that when $p_1, \dots, p_d$ are distinct, then $\vec{V}_{\vec{p}}(\vec{\lambda})$ is an irreducible 
$\g[t]$-module. 
\end{definition}

\begin{example}\label{exam7}
  Take a basis $\{e_{i}\}$  and the dual basis $\{e^{i}\}$ of $\mathfrak{g}$ under an invariant  non-degenerate symmetric form $\langle\, , \,\rangle$ on
$\g$. Recall the Casimir element 
$$
\Omega := \sum\limits_{i} e_{i}\otimes e^{i}\in \left[\mathfrak{g}^{\otimes 2}\right]^{\mathfrak{g}}.
$$
Then, by Lemma \ref{lem6}, for any $P,Q\in A$, $\Omega (P,Q)\in [\g\otimes A]^\g$.

For any $\vec{p}=(p_1, p_2) \in \C^2$, consider the evaluation module $\vec{V}_{\vec{p}}(\vec{\lambda})$, where $\vec{\lambda}=(\lambda_1, \lambda_2)$ is a pair of dominant integral weights. Let
\begin{align*}
P(p_{1}) &= w_1, P(p_{2})=w_2\\[4pt]
Q(p_{1}) &= z_{1}, Q(p_{2})=z_{2}.
\end{align*}
\end{example}

\begin{lemma} \label{casimir} For any $v$ in the $\g$-isotypic component of $V(\lambda_1) \otimes V(\lambda_2)$ of highest weight $\mu$,
\begin{align} \Omega (P,Q)& (v)=\notag\\
& \left(w_1z_1C_{\lambda_1} + w_2z_2C_{\lambda_2} +\frac{w_1z_2+w_2z_1}{2}(C_\mu-C_{\lambda_1}-C_{\lambda_2})\right) v,
\end{align}
where $C_\mu$ denotes the scalar by which $\Omega$ acts on $V(\mu)$.
\end{lemma}
\begin{proof} 
For $v_{1}\in V (\lambda_1)$ and $v_{2}\in V (\lambda_2)$,
\begin{align}
\left(\sum\limits_{i} e_{i}(P)\cdot e^{i}(Q)\right)&\cdot (v_1 \otimes v_2)= 
w_1z_1C_{\lambda_1} (v_1\otimes v_2)+ w_2z_2C_{\lambda_2} (v_1\otimes v_2)\notag\\ &+  w_1z_2\sum_i\, e_i\cdot v_1\otimes e^i\cdot v_2+w_2z_1 \sum_i\, e^i\cdot v_1\otimes e_i\cdot v_2.\label{eq101}
\end{align}
Taking $w_1=w_2=z_1=z_2=1$ in the above, we get
\begin{equation} \label{eq102} \Omega(v_1 \otimes v_2)=C_{\lambda_1} (v_1\otimes v_2)+ C_{\lambda_2} (v_1\otimes v_2)+  \sum_i\, e_i\cdot v_1\otimes e^i\cdot v_2+ \sum_i\, e^i\cdot v_1\otimes e_i\cdot v_2.
\end{equation}
Further, it is easy to see that 
\begin{equation} \label{eq103} \sum_i\, e^i\cdot v_1\otimes e_i\cdot v_2 = \sum_i\, e_i\cdot v_1\otimes e^i\cdot v_2.
\end{equation}
Thus, combining the equations \eqref{eq101} - \eqref{eq103}, we get 
$$\left(\sum\limits_{i} e_{i}(P)\cdot e^{i}(Q)\right)\cdot v=\left(w_1z_1C_{\lambda_1} + w_2z_2C_{\lambda_2} +\frac{w_1z_2+w_2z_1}{2}(C_\mu-C_{\lambda_1}-C_{\lambda_2})\right) v,$$
for $v$ in the $\g$-isotypic component of $V(\lambda_1) \otimes V(\lambda_2)$ of highest weight $\mu$. This proves the lemma.
\end{proof}

\begin{example} 
Take any simple Lie algebra $\mathfrak{g}$ and any
$$
\theta \in \left[\wedge^{k}(\mathfrak{g})\right]^{\mathfrak{g}}\subset \left[\otimes^{k}\mathfrak{g}\right]^{\mathfrak{g}},
$$
where we think of $\wedge^{k}(\mathfrak{g})$ as the subspace of $\otimes^{k}\mathfrak{g}$ consisting of alternating tensors. Recall that
$$
\left[\wedge^{\bullet}(\mathfrak{g})\right]^{\mathfrak{g}}\simeq \wedge (\theta_{1},\dots,\theta_{\ell}),
$$
where $\ell:= \rank\mathfrak{g}$ and $\theta_{j}\in \left[\wedge^{2b_{j}+1}(\mathfrak{g})\right]^\g$ are primitive generators. Here  $1=b_{1} \leq \dots \leq b_{\ell}$ are the exponents of $\mathfrak{g}$.
This gives rise to $\theta(P_{1},\dots,P_{k})\in \left[U(\mathfrak{g}\otimes A\right]^{\mathfrak{g}}$, for any $P_1, \dots P_k\in A.$

Similarly, 
$$\left[S^{k}(\mathfrak{g})\right]^\g \subset \left[\otimes^{k}\mathfrak{g}\right]^\g, $$ 
where we think of $S^{k}(\mathfrak{g})$ as the subspace of $\otimes^{k}\mathfrak{g}$ consisting of symmetric tensors. Recall that 
$\left[S^{\bullet}(\mathfrak{g})\right]^\g$ is a polynomial algebra generated by certain homogeneous elements $\{\delta_j\}_{1\leq j \leq \ell}$, where
$\delta_j\in \left[S^{b_j+1}(\mathfrak{g})\right]^\g$.

\end{example}

\begin{proposition}\label{prop8}
For any finite dimensional reductive Lie algebra $\mathfrak{g}$, the subalgebra $\left[U(\mathfrak{g}[t])\right]^{\mathfrak{g}}$ is 
spanned (over $\C$) by $\{\theta(n_{1},\dots,n_{k})\}$, where $\theta$ runs over a homogeneous basis of $\left[T(\g)\right]^\g$ and, for $k=$ deg $\theta$, $n_1 \leq \dots \leq n_k$ runs over non-negative integers. 

Further, the subalgebra $\left[U(\mathfrak{g}[t])\right]^{\mathfrak{g}}$ is generated (as an algebra) by $\{\theta(n_{1},\dots,n_{k})\}$, where $\theta$ runs over a set of homogeneous algebra generators of $\left[T(\mathfrak{g})\right]^{\mathfrak{g}}$ and, for $k=\text{\rm deg~}\theta$, $n_{1}\leq\dots \leq n_{k}$ runs over non-negative integers.

(Here $\theta(n_{1},\dots,n_{k})$ denotes $\theta(t^{n_{1}},\dots, t^{n_{k}}).$)
\end{proposition}

\begin{proof} As earlier, consider the surjective algebra homomorphism:
\begin{equation}
\pi: T(\mathfrak{g}[t])\twoheadrightarrow U(\mathfrak{g}[t]).\label{eq13}
\end{equation}
Now, $\mathfrak{g}[t]=\bigoplus^{\infty}_{n=0}\,\mathfrak{g}(n)$, where $\g(n):= \g\otimes t^n$.  Hence,
\begin{equation}
T(\mathfrak{g}[t])=\bigoplus_{k\geq 0} \bigoplus\limits_{n_{1},\dots,n_{k}\in Z_{+}} \mathfrak{g}(n_{1})\otimes\cdots\otimes \mathfrak{g}(n_{k})\simeq \bigoplus_{k\geq 0} \bigoplus\limits_{n_{1},\dots,n_{k}\in Z_{+}} \mathfrak{g}^{\otimes k}[n_{1},\dots,n_{k}],\text{~~ as $\mathfrak{g}$-modules,}\label{eq14}
\end{equation}
where $\mathfrak{g}^{\otimes k}[n_{1},\dots,n_{k}]$ simply means $\g^{\otimes k}$ as a $\g$-module, but as a subset of $T(\mathfrak{g}[t])$, it is defined to be $\g(n_1) \otimes \dots \otimes \g(n_k)$.

By \eqref{eq13}, we get
$$
\left[T(\mathfrak{g}[t])\right]^{\mathfrak{g}}\twoheadrightarrow \left[U(\mathfrak{g}[t])\right]^{\mathfrak{g}}.
$$
From \eqref{eq14}, we get
$$
\left[T(\mathfrak{g}[t])\right]^{\mathfrak{g}}\simeq \bigoplus_{k\geq 0} \bigoplus\limits_{n_{1},\dots,n_{k}\in \mathbb{Z}_{+}} \left[\mathfrak{g}^{\otimes k}\right]^{\mathfrak{g}} [n_{1},\dots,n_{k}].
$$
From this we see that $\{\hat{\theta} (n_{1},\dots,n_{k})\}$ spans (resp. generates) the algebra $\left[T(\mathfrak{g}[t])\right]^{\mathfrak{g}}$,  where $\theta$ runs over a homogeneous basis (resp. homogeneous algebra generators) of $\left[T(\g)\right]^\g$ and, for $k=$ deg $\theta$, $n_1, \dots ,  n_k$ runs over non-negative integers. (Here $\hat{\theta} (n_{1},\dots,n_{k})$ denotes the element in $\left[T(\mathfrak{g}[t])\right]^{\mathfrak{g}}$ as in the proof of Lemma \ref{lem6}.) From the surjectivity of the algebra homomorphism $\pi$, we get that
$\{{\theta} (n_{1},\dots,n_{k})\}$ spans in the first case (resp. generates in the second case), where $n_1, \dots n_k$ runs over non-negative integers. 
Now, we can restrict to   $0\leq n_1 \leq  \dots \leq n_k$ in the first case, which is easily seen from the commutation relation in $U(\g[t])$.
\end{proof}

Let $\vec{\lambda}= (\lambda_1, \dots, \lambda_d)$ be any tuple of dominant integral weights and let  $\vec{p}= (p_1, \dots, p_d)$ be any tuple of {\it distinct} points in $\C$. Then, the evaluation module $\vec{V}_{\vec{p}}(\vec{\lambda})$ is an irreducible $\g[t]$-module (cf. Definition \ref{defi6}).
Decompose $\vec{V}_{\vec{p}}(\vec{\lambda})$  into its isotypic components (as a $\g$-module):
$$
\vec{V}_{\vec{p}}(\vec{\lambda}) =\oplus \vec{V}_{\vec{p}}(\vec{\lambda}) [\mu],
$$
where $\vec{V}_{\vec{p}}(\vec{\lambda}) [\mu]$ denotes the  isotypic component corresponding to the highest weight $\mu$. 

Clearly, the action of $\g$ commutes with the action of $\left[U(\g[t])\right]^\g$  on $\vec{V}_{\vec{p}}(\vec{\lambda})$. Thus, we get an action of 
$\g\times \left[U(\g[t])\right]^\g$ on $\vec{V}_{\vec{p}}(\vec{\lambda})$  stabilizing each isotypic component $\vec{V}_{\vec{p}}(\vec{\lambda}) [\mu].$ 

The following theorem is one of our principal results of the paper. In view of Remark \ref{remark11}, the following theorem in the case of $\g=gl(n)$ was stated as an open problem in [R, Remark 4.9].
\begin{theorem}\label{thm9}
Let $\g$ be  a reductive Lie algebra. With the notation as above, each isotypic component $\vec{V}_{\vec{p}}(\vec{\lambda}) [\mu]$ is an irreducible module for $\mathfrak{g}\times \left[U(\g[t])\right]^\g$.
\end{theorem}

\begin{proof} Choose a Borel subalgebra $\mathfrak{b} $ of $\g$ and let $W[\mu]^+$ be the 
 $\mathfrak{b}$-eigen subspace  of $W [\mu] := \vec{V}_{\vec{p}}(\vec{\lambda}) [\mu].$  Then, $\left[U(\mathfrak{g}[t])\right]^{\mathfrak{g}}$ acts on $W[\mu]^{+}$. We have a ring homomorphism coming from the representation $W :=\vec{V}_{\vec{p}}(\vec{\lambda})$:
$$
\varphi : U(\mathfrak{g}[t])\to \End_{\mathbb{C}}(W).
$$

Since $W$ is an irreducible $\mathfrak{g}[t]$-module, by Burnside's theorem (cf. [L, Corollary 1, Chapter XVII, \S3])
$\varphi$ is surjective. Taking the $\mathfrak{g}$-invariants, we get a surjective homomorphism
\begin{equation}\label{eqnew}
\varphi^{o}: \left[U(\mathfrak{g}[t])\right]^{\mathfrak{g}}\twoheadrightarrow \End_{\mathfrak{g}}(W)\simeq \prod\limits_{\mu}\End_{\mathfrak{g}}(W[\mu]),
\end{equation}
where $\End_{\mathfrak{g}}(W)$ is the space of $\mathfrak{g}$-module endomorphisms of $W$. On projection to $\End_{\mathfrak{g}}(W[\mu])$,  we get a surjective morphism
$$
\varphi^{o}_{\mu} : \left[U(\mathfrak{g}[t])\right]^{\mathfrak{g}}\twoheadrightarrow \End_{\mathfrak{g}}(W[\mu])\simeq \End_{\mathbb{C}}(W[\mu]^{+}).
$$
In particular, $W[\mu]^{+}$ is an irreducible module under the action of $\left[U(\mathfrak{g}[t])\right]^{\mathfrak{g}}$. From this the theorem follows.
\end{proof}

\begin{remark} Since $ \left[U(\mathfrak{g})\right]^{\mathfrak{g}} \subset  \left[U(\mathfrak{g}[t])\right]^{\mathfrak{g}}$, the irreducible 
$\left[U(\mathfrak{g}[t])\right]^{\mathfrak{g}}$-modules $W[\mu]^+$ are mutually inequivalent as $\mu$ ranges over the highest weights of the components 
of   $\vec{V}_{\vec{p}}(\vec{\lambda})$ by Harish-Chandra's theorem (cf. [H, Theorem 23.3]). 
\end{remark}

\section{Determination of $ \left[U(\mathfrak{g})\right]^{\mathfrak{g}}$ for $\g=gl(n)$}\label{sec3}

By Proposition \ref{prop8}, it suffices to determine $\left[T(\g)\right]^\g$.

Let $V =\C^n$ with the standard basis $\{e_1, \dots, e_n\}$ and the standard representation of $gl(n)$. Let $\{e_1^*, \dots, e_n^*\}$ be the dual basis
of $V^*$. Recall the isomorphism:
$$\beta: V^*\otimes V \simeq \End V = gl(n),$$
where $\left(\beta (f\otimes v)\right)(w) = f(w)v$, for 
$v,w \in V, f\in V^*.$ Moreover, the tensor product action of $gl(n)$ on $V^*\otimes V$ corresponds  (under the identification $\beta$) to the adjoint action on $gl(n)$. Thus,
$$\left[\g^{\otimes k}\right]^\g \simeq \left[(V^*\otimes V)^{\otimes k}\right]^\g.$$
By the tensor version of the {\it First Fundamental Theorem} (FFT) (cf. [GW, Theorem 5.3.1]), we get that $\{\theta_\sigma\}_{\sigma \in \Sigma_k}$ spans
$\left[\g^{\otimes k}\right]^\g $,
where 
\begin{align}\theta_\sigma &:= \sum_{1 \leq i_1, \dots, i_k\leq n}\, (e^*_{i_{\sigma (1)}}\otimes e_{i_1})\otimes \dots \otimes (e^*_{i_{\sigma (k)}}\otimes e_{i_k})\notag\\
&=\sum_{1 \leq i_1, \dots, i_k\leq n}\,E_{i_1, i_{\sigma (1)}}\otimes \dots \otimes E_{i_k, i_{\sigma (k)}}, \,\,\,\text{under the isomorphism}\,\, \beta,
\end{align}
where $E_{i,j}$ is the $n\times n$-matrix with $(i,j)-$th entry $1$ and all other entries zero.

Thus, as a corollary of Proposition \ref{prop8}, we get:
\begin{corollary} Let $\g=gl(n)$. Then, $\left[U(\g[t])\right]^\g$ is spanned by
$$\bigcup_{k\geq 0}\,\bigcup_{0\leq n_1\leq \dots \leq n_k}\, \{\sum_{1 \leq i_j\leq n}\, E_{i_1, i_{\sigma (1)}}(n_1)\dots  E_{i_k, i_{\sigma (k)}}(n_k)\}_{\sigma\in \Sigma_k}.$$
\end{corollary}
\begin{remark} \label{remark11} From the explicit description of the standard generators of the center of $U(gl(n))$, the cycle $\sigma_k:=(1,2, \dots, k)$ plays a special role (also see [R, \S4.2]). In this case
$$\theta_{\sigma_k} = \sum_{1 \leq i_j\leq n}\, E_{i_1, i_2}\otimes E_{i_2, i_3}\otimes\dots  \otimes E_{i_k, i_1}.$$
Considering the cycle decomposition of permutations in $\Sigma_k$, it is easy to see (by the above corollary) that the elements
$$\{\theta_{\sigma_k}(n_1, \dots, n_k): 1 \leq k \,\,\text{and}\,\, n_1, \dots , n_k \geq 0\}$$
generate the algebra $\left[U(gl(n)[t])\right]^{gl(n)}$. This provides an affirmative answer to a question of Rao [R, Problem 4.5 (1)]. Observe that in the above set of generators, we take $n_1, \dots , n_k$ to vary over $\mathbb{Z}_{\geq 0}$. 

If we take the subset $\{\theta_{\sigma_k}(n_1, \dots, n_k): 1 \leq k \,\,\text{and}\,\, 0\leq n_1\leq  \dots \leq n_k \}$, then it does {\it not} generate the algebra $\left[U(gl(n)[t])\right]^{gl(n)}$.
\end{remark}

\section{Realizing Schur-Weyl duality intertwiners in terms of current algebras}

As in the last section, let $V =\C^n$ be the standard representation of $\g=gl(n)$. For any positive integer $k$, the symmetric group $\Sigma_k$ acts on the tensor product $V^{\otimes k}$ by permuting the factors. Clearly, this action of $\Sigma_k$ commutes with the tensor product action of $gl(n)$
on $V^{\otimes k}$. 
Thus, we have an algebra homomorphism:
$$\Phi: \C[\Sigma_k] \to \End_\g(V^{\otimes k}),$$
where $\End_\g(V^{\otimes k})$ denotes the space of $\g$-module endomorphisms of $V^{\otimes k}$.

The content of the Schur-Weyl duality is the following result (cf. [GW, \S4.2.4]).
\begin{theorem} The above map $\Phi$ is an (algebra) isomorphism.
\end{theorem}

Choose  $\vec{p}=(p_1, \dots , p_k) \in \C^k$ such that $p_i$'s are distinct. By the proof of Theorem \ref{thm9} (see \eqref{eqnew}), there is a surjective algebra homomorphism (for $\g=gl(n)$)
$$\varphi^{o}: \left[U(\mathfrak{g}[t])\right]^{\mathfrak{g}}\twoheadrightarrow \End_{\mathfrak{g}}(V^{\otimes k}) \simeq  \C[\Sigma_k] ,$$
where the last identification is via $\Phi$. Thus, we get a surjective algebra homomorphism
$$\Xi: \left[U(\mathfrak{g}[t])\right]^{\mathfrak{g}}\twoheadrightarrow  \C[\Sigma_k] .$$
The following result is easy to prove using the definition of the map $\Xi$.
\begin{proposition} \label{schur} For any reflection $\tau = (r, s) \in \Sigma_k$, define the polynomials:
$$P_\tau = (t-p_r+1)\cdot\prod_{d\neq r}\,\frac{(t-p_d)}{(p_r-p_d)},\,\,\, Q_\tau = (t-p_s+1)\cdot\prod_{d\neq s}\,\frac{(t-p_d)}{(p_s-p_d)}.$$
Then,
$$ \Xi\left(\sum_{1\leq i, j\leq n}\,E_{i,j}(P_\tau)\cdot E_{j,i}(Q_\tau)\right)= \tau.$$
\end{proposition}

\section{Determination of $ \left[U(\mathfrak{g})\right]^{\g}$ for the symplectic Lie algebra $\g={\s}p (2n)$} 

Again we need to determine $\left[\g^{\otimes k}\right]^\g$ for $\g={\s}p (2n)$.

 Let $V=\C^{2n}$ be equipped with the
nondegenerate symplectic form $\langle \,,\,\rangle$ so that its matrix
$\bigl(\langle e_i,e_j\rangle\bigr)_{1\leq i,j \leq 2n}$ in the
standard basis $\{e_1,\dots, e_{2n}\}$ is given by
\begin{equation*}
\hat{J}=\left(\begin{array}{cc}
0&J\\
-J&0
\end{array}\right),
\end{equation*}
where $J$ is the anti-diagonal matrix $(1,\dots,1)$ of size $n$. Let
$$\Sp(2n):=\{g\in \SL(2n):
g \,\text{leaves  the form}\, \langle \,,\,\rangle \,\text{invariant}\}$$ be the associated
symplectic group and ${\s}p (2n)$ its Lie algebra.  Thus, $\Sp(2n)$ has defining representation in $V$. The form  $\langle \,,\,\rangle $ allows us  to identify 
$V \simeq V^*$ ($v \mapsto f_v(w)= \langle v, w\rangle$ for $v,w \in V$). Then, there is an identification 
$$S^2(V) \simeq {\s}p (2n) ,$$
such that the standard action of  ${\s}p (2n)$ on the left  corresponds to the adjoint action on the right. Hence, for 
$\g={\s}p (2n)$,
$$\left[\g^{\otimes k}\right]^\g \simeq \left[\otimes^k(S^2(V))\right]^{{\s}p (2n)}.$$
By the tensor version of FFT for $\Sp(2n)$ (cf. [GW, Theorem 5.3.3]), we get that 
$\left[\otimes^k(V^{\otimes 2})\right]^{{\s}p (2n)}$ is spanned by 
$$\{\Theta_\sigma :=\sum_{1 \leq i_1, i_3, \dots, i_{2k-1} \leq 2n}\,(e_{i_{\sigma(1)}}\otimes e_{i_{\sigma(2)}})\otimes \dots \otimes (e_{i_{\sigma(2k-1)}}\otimes e_{i_{\sigma(2k)}})\}_{\sigma\in \Sigma_{2k}},$$
where we require $e_{i_{2j}} =s(i_{2j-1}) e_{2n+1 -i_{2j-1}}$ for $1\leq j\leq k$ and $s:\{1,2, \dots, 2n\} \to \{\pm 1\}$ is the function:
\begin{align} \label{0}
s (i)&=1, \,\,\,\text{if}\,\,\, 1\leq i\leq n\notag\\
 &=-1, \,\,\,\text{if}\,\,\, i>n.\notag
 \end{align}
The standard symmetrization $V^{\otimes 2} \to S^2(V)$ gives rise to a surjective map 
$$\gamma: \left[\otimes^k (V^{\otimes 2})\right]^{{\s}p (2n)} \twoheadrightarrow  \left[\otimes^k(S^2(V))\right]^{{\s}p (2n)} \simeq \left[\g^{\otimes k}\right]^\g.$$
Under this identification, we get the element 
\begin{align*}\gamma (\Theta_\sigma)=\frac{1}{2}\sum_{1 \leq i_1, i_3, \dots, i_{2k-1} \leq 2n}\, &\left(s(i_{\sigma(1)})E_{i_{\sigma(2)}, 2n+1-i_{\sigma(1)}}+s(i_{\sigma(2)})E_{i_{\sigma(1)}, 2n+1-i_{\sigma(2)}}\right) \otimes \dots\otimes \\
&\left(s(i_{\sigma(2k-1)})E_{i_{\sigma(2k)}, 2n+1-i_{\sigma(2k-1)}}+s(i_{\sigma(2k)})E_{i_{\sigma(2k-1)}, 2n+1-i_{\sigma(2k)}}\right),
\end{align*}
where the square matrix $E_{i,j}$ is as defined in Section \ref{sec3}.

Thus, as a corollary of Proposition \ref{prop8}, we get:
\begin{corollary} Let $\g={\s}p(2n)$. Then, $\left[U(\g[t])\right]^\g$ is spanned by
\begin{align*} 
\cup_{k\geq 0}\,\cup_{0\leq n_1\leq \dots \leq n_k}\, &\{\sum_{1 \leq i_1, i_3, \dots, i_{2k-1} \leq 2n}\, \left(s(i_{\sigma(1)})E_{i_{\sigma(2)}, 2n+1-i_{\sigma(1)}}+s(i_{\sigma(2)})E_{i_{\sigma(1)}, 2n+1-i_{\sigma(2)}}\right) (n_1)\otimes \dots \\
&\otimes \left(s(i_{\sigma(2k-1)})E_{i_{\sigma(2k)}, 2n+1-i_{\sigma(2k-1)}}+s(i_{\sigma(2k)})E_{i_{\sigma(2k-1)}, 2n+1-i_{\sigma(2k)}}\right)(n_k)\}_{\sigma\in \Sigma_{2k}},
\end{align*}
where we require $e_{i_{2j}} =s(i_{2j-1}) e_{2n+1 -i_{2j-1}}$ for $1\leq j\leq k$, which gives the corresponding constraint on $E_{p,q}=e_q^*\otimes e_p$.
\end{corollary}

\section{Determination of $ \left[U(\mathfrak{s}o(n))\right]^{\SO(n)}$} 

Let $V=\C^{n}$ be equipped with the
nondegenerate symmetric form $\langle \,,\,\rangle$ so that the standard basis  $\{e_1,\dots, e_{n}\}$ is orthonormal.
 Let
$$  \SO(n):=\{g\in \GL(n):
g \,\text{leaves  the form}\, \langle \,,\,\rangle \,\text{invariant}\}$$ be the associated
orthogonal group and ${\s}o (n)$ its Lie algebra.  Thus, $\SO(n)$ has defining representation in $V$. The form  $\langle \,,\,\rangle $ allows us  to identify 
$V \simeq V^*$ ($v \mapsto f_v(w)= \langle v, w\rangle$ for $v,w \in V$). Then, there is an identification 
$$\wedge^2(V) \simeq {\s}o (n),$$
such that the standard action of  $\SO (n)$ on the left  corresponds to the adjoint action on the right. Hence, for 
$\g={\s}o (n)$,
$$\left[{\s}o(n)^{\otimes k}\right]^{\SO(n)} \simeq \left[\otimes^k(\wedge^2(V))\right]^{\SO(n)}.$$
By the tensor version of FFT for $\SO(n)$ (cf. [GW, Theorem 5.3.3]), we get that 
$\left[\otimes^k(V^{\otimes 2})\right]^{\SO(n)}$ is spanned by 
$$\{\Psi_\sigma :=\sum_{1 \leq i_1, i_3, \dots, i_{2k-1} \leq n}\,(e_{i_{\sigma(1)}}\otimes e_{i_{\sigma(2)}})\otimes \dots \otimes (e_{i_{\sigma(2k-1)}}\otimes e_{i_{\sigma(2k)}})\}_{\sigma\in \Sigma_{2k}},$$
where we require $i_{2j}=i_{2j-1}$ for $1 \leq j \leq k$.

The standard anti-symmetrization $V^{\otimes 2} \to \wedge^2(V)$ gives rise to a surjective map 
$$\delta: \left[\otimes^k (V^{\otimes 2})\right]^{ \SO(n)} \twoheadrightarrow  \left[\otimes^k(\wedge^2(V))\right]^{\SO(n)} \simeq        
\left[{\s}o(n)^{\otimes k}\right]^{\SO(n)} .$$
Under this identification, we get the element 
\begin{align*}\delta (\Psi_\sigma)=- \frac{1}{2}\sum_{1 \leq i_1, i_3, \dots, i_{2k-1} \leq n}\, &\left(E_{i_{\sigma(1)}, i_{\sigma(2)}}-E_{i_{\sigma(2)}, i_{\sigma(1)}}\right) \otimes \dots\otimes \\
&\left(E_{i_{\sigma(2k -1)}, i_{\sigma(2k)}} - E_{i_{\sigma(2k)}, i_{\sigma(2k-1)}}\right),
\end{align*}
where the square matrix $E_{i,j}$ is as defined in Section \ref{sec3}.

Thus, as a corollary of Proposition \ref{prop8}, we get:
\begin{corollary} Let $\g={\s}o(n)$. Then, $\left[U(\g[t])\right]^{\SO(n)}$ is spanned by
\begin{align*} 
\bigcup_{k\geq 0}\,\bigcup_{0\leq n_1\leq \dots \leq n_k}\, \{\sum_{1 \leq i_1, i_3, \dots, i_{2k-1} \leq n}\, &
\left(E_{i_{\sigma(1)}, i_{\sigma(2)}}-E_{i_{\sigma(2)}, i_{\sigma(1)}}\right) (n_1)\otimes \dots\otimes \\
&\left(E_{i_{\sigma(2k -1)}, i_{\sigma(2k)}} - E_{i_{\sigma(2k)}, i_{\sigma(2k-1)}}\right)(n_k)\}_{\sigma\in \Sigma_{2k}},
\end{align*}
 where we require $i_{2j}=i_{2j-1}$ for $1 \leq j \leq k$.
\end{corollary}
\begin{remark} For $n$ odd, observe that 
$$\left[{\s}o(n)^{\otimes k}\right]^{\SO(n)} \simeq \left[{\s}o(n)^{\otimes k}\right]^{{\s}o (n)}.$$
\end{remark}

\vskip6ex
\noindent
S.K.: Department of Mathematics, University of North Carolina,
Chapel Hill, NC 27599-3250, USA (email: shrawan$@$email.unc.edu)

\end{document}